\documentclass{amsart}
\usepackage[dvips]{color}
\usepackage{graphicx}

\usepackage{latexsym}
\usepackage{amssymb}
\usepackage{amsmath}
\usepackage{amsthm}
\usepackage{amscd}
\theoremstyle{plain}
\newtheorem{thm}{Theorem}
\newtheorem{lem}{Lemma}
\newtheorem{Bob}{Definition}
\newtheorem{prop}{Proposition}

\theoremstyle{definition}
\newtheorem{rem}{Remark}

\newtheorem{ques}{Question}
\newcommand{\Leb}{\ensuremath{\lambda}}
\newcommand{\LebP}{\ensuremath{\bold{m}_{Q}}}
\newcommand{\measure}{\ensuremath{\bold{m}_{\mathfrak{R}}}}
\newcommand{\LS}{\ensuremath{\underset{n=1}{\overset{\infty}{\cap}} \, {\underset{i=n}{\overset{\infty}{\cup}}}\,}}
\newcommand{\CLS}{\ensuremath{\underset{n=1}{\overset{\infty}{\cap}} \, \underset{t\geq n}{\cup}\,}}
\newcommand{\LebF}{\ensuremath{\nu_{Q}}}
\newtheorem{cor}{Corollary}
\begin{document}
\title{Shrinking targets for IETs: extending a theorem of Kurzweil}\
\author[J.\ Chaika]{Jon Chaika}

\address{Department of Mathematics, Rice University, Houston, TX~77005, USA}

\email{Jonathan.M.Chaika@rice.edu}
\maketitle

\begin{abstract} Given an IET $T:[0,1) \to [0,1)$ and decreasing sequence of positive real numbers with divergent sum $\textbf{a}= \{a_i\}_{i=1}^{\infty}$ we consider
$$S_T(\textbf{a})=\{(x,y)\in[0,1) \times[0,1): y \in B(T^ix,a_i) \text{ for infinitely many }i\} $$ where 
 $B(x,r)$ is the ball of radius $r$ about $x$. We prove that
\begin{enumerate}
\item for any fixed $\bf a$ for almost every IET $T$ the set  $S_T(\bf{a})$ has full Lebesgue measure.
\item For almost every IET there exists $\bf a$ such that $S_T(\bf a)$ has zero Lebesgue measure.
\item If one restricts to non-increasing sequences of positive real numbers with divergent sum which have the additional property that $ia_i$ is non-increasing then for almost every IET $T$ the set $S_T(\bf{a})$ has full Lebesgue measure for all such sequences.
\end{enumerate}
We prove related results for geodesic flows on translation surfaces and stronger results which treat the measure of every horizontal and vertical line of $S_T(\bf a)$.
\end{abstract}

\vspace{3 mm}

Let $(X, d)$ be a compact metric space and $T \colon X \to X$ be a $\mu$-ergodic map where $\mu$ is a finite Borel measure. It follows from the Birkhoff Ergodic Theorem that $\mu(\LS B(T^i x, \epsilon))=\mu(X)$ for every $\epsilon>0$ and $\mu$-almost every $x$. This means that $\mu$-almost every point $x$ lands in the ball of radius $\epsilon$ centered about $\mu$-almost every point $y$ infinitely often. This is equivalent to stating $\underset{i \to \infty}{\liminf} \, d(T^ix,y)=0$ for $\mu \times \mu$-almost every $(x,y)$. The shrinking target problem seeks to establish quantitative analogues of this; that is, let $\textbf{a}=\{a_i\}_{i=1}^{\infty}$ be a decreasing sequence of positive numbers,  is $\mu(\LS B(T^i x, a_i))=\mu(X)$ for  $\mu$-almost every $x$? The Borel-Cantelli Theorem provides a necessary condition ($\underset{i=1}{\overset{\infty}{\sum}} \mu(B(T^ix,a_i))=\infty $) and therefore shrinking target theorems often take the form of partial converses to the Borel-Cantelli Theorem. Let us recall the Borel-Cantelli Theorem.
\begin{thm} (Borel-Cantelli) Let $\mu$ be a probability measure and $\{A_i\}_{i=1}^{\infty}$ be a sequence of $\mu$-measurable sets.
\begin{enumerate}
\item If $\underset{i=1}{\overset{\infty}{\sum}} \mu(A_i)<\infty$ then $\mu(\LS A_i)=0.$
\item If $\mu(A_i\cap A_j)=\mu(A_i)\mu(A_j)$ for all $i\neq j$ then $\underset{i=1}{\overset{\infty}{\sum}} \mu(A_i)=\infty$ implies
$\mu(\LS A_i)=1$.
\end{enumerate}
\end{thm}
The converse to the first part of the Borel-Cantelli Theorem is false, but the second part tells us that the converse holds when the sets are independent. In this way shrinking target properties are partial converses; they provide that divergence of the measures of the sets considered implies that the set limsup has full measure under additional assumptions on the sets.

 Given $(X,T)$, a sequence of measurable sets $A_1,A_2,... \subset X$ is said to be \emph{Borel-Cantelli}
  if $\mu$-almost every $x$ satisfies $T^ix \in  A_i$ for infinitely many $i$.
  (Equivalently, $\LS T^{-i}(A_i)$ has full $\mu$ measure.)
  $(X,d,T)$ is said to satisfy the \emph{Monotone Shrinking Target Property} (MSTP)
  if the sequence of measurable sets given by ${A_i=B(y,a_i)}$ is Borel-Cantelli for any $y$,
   and any decreasing sequence of positive numbers $\{a_i\}_{i=1}^{\infty}$~ with
   $\underset{i=1}{\overset{\infty}{\sum}} \mu(B(y,a_i))= \infty$.
   We refer the reader interested in the Monotone Shrinking Target Property to the survey \cite{ja}
    and the accessible paper \cite{fayad}, which reproves Kurzweil's result that rotations by badly approximable
     numbers are exactly the rotations satisfying MSTP and also provides the first example of a mixing system that does not satisfy MSTP. 

In the 1950's 
 J. Kurzweil  established a result for rotations
  which largely motivates this paper. Define $R_{\alpha} \colon [0,1) \to[0,1)$ to be $R_{\alpha}(x)=x+\alpha-\lfloor x+\alpha \rfloor$, rotation by $\alpha$.
\begin{thm} (Kurzweil \cite{kurz})\label{startthm} For any decreasing sequence of positive real numbers $\{a_i\}_{i=1}^{\infty}$ with divergent sum there exists $\mathcal{V} \subset [0,1)$, a full measure set of $\alpha$, such that for all ${\alpha \in \mathcal{V}} $ we have
$$\Leb\left(\LS B(R_{\alpha}^i(x),a_i)\right)=1 $$
 for every $x$.
 
  On the other hand,
$$\Leb\left(\LS B(R_{\alpha}^i(x),a_i)\right)=1 $$
 for every $x$ and every decreasing sequence of positive real numbers $\{a_i\}_{i=1}^{\infty}$~  with divergent sum iff $\alpha$ is badly approximable.
\end{thm}
Recall that  $\alpha$ is badly approximable if the terms in its continued fraction expansion  are uniformly bounded. Because badly approximable numbers are a (meager) set of measure 0, Lebesgue almost every $\alpha$ does not satisfy MSTP. Additionally, 
a part of Kurzweil's Theorem treats fixing a sequence and making a full measure statement rather than addressing all sequences satisfying a property at once.


This paper extends Kurzweil's results to interval exchange transformations (IETs) and geodesic flows on translation surfaces. The first section establishes terminology and states the theorems. The main results of this paper are Corollary~ \ref{strong kurz for IET} and Theorems \ref{main theorem}, \ref{rigidthm} and \ref{full measure thm}.


\section{Terminology and statement of results}
\begin{Bob} Given $L=(l_1,l_2,...,l_d)$
where $l_i \geq 0$, we obtain $d$ sub-intervals of the
interval $[0,\underset{i=1}{\overset{d}{\sum}} l_i)$: $$I_1=[0,l_1) ,
I_2=[l_1,l_1+l_2),...,I_d=[l_1+...+l_{d-1},l_1+...+l_{d-1}+l_d).$$ Given
 a permutation $\pi$ on  the set $\{1,2,...,d\}$, we obtain a d-\emph{Interval Exchange Transformation} (IET)  $ T \colon [0,\underset{i=1}{\overset{d}{\sum}} l_i) \to
 [0,\underset{i=1}{\overset{d}{\sum}} l_i)$ which exchanges the intervals $I_i$ according to $\pi$. That is, if $x \in I_j$ then $$T(x)= x - \underset{k<j}{\sum} l_k +\underset{\pi(k')<\pi(j)}{\sum} l_{k'}.$$
\end{Bob}
If $T$ is an IET, let $L(T)$ denote the length vector of $T$ and $\pi(T)$ denote the permutation of $T$. The IET with length vector $L$ and permutation $\pi$ is denoted $S_{L,\pi}$.
It is often convenient to restrict one's attention to IETs mapping from $[0,1)$ to $[0,1)$.
 In this case, IETs with a fixed permutation on $\{1,2,...,d\}$ are parametrized by the standard simplex in $\mathbb{R}^d$,
  $\Delta_d= \{(l_1,...,l_d): l_i \geq 0, \sum l_i=1\}$. We will denote Lebesgue measure on $\Delta_d$ by $\bold{m}_d$. 
We will denote Lebesgue measure on the unit interval (where unit length IETs act) by $\Leb$.
 A permutation on $\{1,...,d\}$ is \emph{irreducible} if $\pi(\{1,...,k\}) \neq \{1,...,k\}$ for any $k<d$. These are the permutations that yield IETs with dense orbits \cite{IET} and thus are the interesting IETs from the standpoint of shrinking target properties. The term almost every IET refers to Lebesgue measure on the disjoint union of all the simplices corresponding to irreducible permutations (which we view as the parameterizing space of all the IETs we are considering). 
The following shrinking target results are known for IETs.

\begin{thm}\label{start} (Boshernitzan and Chaika \cite{BCIET})
If $T$ is an IET that is ergodic with respect to some measure $\mu$ then
$$\mu\left(\LS B(T^ix,\frac {\epsilon} i)\right)=1$$
 for any $\epsilon>0$ and $\mu$-almost every $x$.
Moreover, if ${\underset{i \to \infty}{\lim} \, ia_i=0}$
 then there exists an irrational rotation $R$ such that
$$\Leb\left(\LS B(R^ix,a_i)\right)=0$$ for every $x$.
 Additionally there exists a 4-IET $T_0$,
 minimal, but not ergodic with respect to $\Leb$ such that $$\Leb\left(\LS B(T_0^ix,\frac 1 i)\right)<1$$ for a positive measure set of $x$.
\end{thm}
The following result is known for shrinking targets about a point and is strengthened by Corollary \ref{strong kurz for IET}.
\begin{thm}(Athreya and Ulcigrai \cite{ja}) \label{AU} Given $y \in [0,1)$ almost every IET $T$ satisfies the property that $$\Leb\left(\LS T^{-i}B(y,\frac c i)\right)=1$$ for some $c$ depending on $T$.
\end{thm}
Another related result,
\begin{thm} (Kim and Marmi \cite{log iet}) Given an IET $T$ let $${\tau_r(x,y)=\min \{n>0: |T^nx-y|<r\}}.$$ For almost every IET $T$, ${\underset{ r \to 0^+}{\lim} \frac{ \log(\tau_r(x,y))}{-\log r}= 1}$ for almost every $x$.
\end{thm}
The $\liminf$ part of the statement was established by Galatolo
\cite{galatolo}.

A homogeneous result has recently been proven.
\begin{thm} \label{march} (Marchese \cite{LM}) Let $\{a_i\}_{i=1}^{\infty}$ be a decreasing sequence with divergent sum and with the additional property that $\{ia_i\}_{i=1}^{\infty}$ is decreasing. For almost every IET $T$ $$\delta \in \LS B(T^i(\delta'),a_i))$$ where $\delta$ and $\delta'$ are any discontinuities of $T$.
\end{thm}
All of the above results also have interpretations for the other dynamical system we are concerned with: unit speed flow on translation surfaces. For an introduction to translation surfaces see \cite{t surf} or \cite{survey}. Section 1 of \cite{bil surv} provides a nice treatment of a special case.
\begin{Bob} A \emph{ translation surface}  $Q$ is the finite union of polygons $P_1,...,P_r$ such that
\begin{enumerate}
\item the sides of the polygons are oriented so that the interior lies to the left
\item each side is identified to exactly one parallel side of the same length. They are glued together in an opposite orientation by parallel translation.
\end{enumerate}
\end{Bob}
This definition appears in
\cite[Definition 4]{t surf}.
In flat surfaces distance and a 2-dimensional volume $\LebF$ make sense because they make sense in each polygon. Direction makes sense because of the gluings.  
Let us assume that there is a fixed horizontal direction. $F_{\theta}^t$ denotes flow with unit speed in direction $2\pi \theta$ to the horizontal. Straight line flows with unit speed on $Q$ are parametrized by $[0,1)$.  

We now present slightly different idea. Fix a translation surface $Q$ and a line segment $\bar v$. The flow on $Q$ in every direction not parallel to $\bar v$ gives an interval exchange transformation on $\bar v$ by the first return map (see \cite[Section 5.1]{survey} for a discussion in a survey paper). 
Let us choose the direction of $\bar v$ to be the horizontal. In this way we obtain a one parameter family of flows on $Q$, $\{F_{\theta}^t\}_{\theta \in (0, 1)}$ and a corresponding one parameter family of IETs on $\bar v$, $\{T_{\theta}\}_{\theta \in (0,1)}$. As is the case in Kurzweil's Theorem we will make statements about almost every transformation. 
  For readability reasons we will denote Lebesgue measure on $[0,1)$ by $\LebP$ when it is the parameterizing space of transformations. Recall that when $[0,1)$ is the space a transformation acts on we denote Lebesgue measure by $\Leb$.

A specific case of a flat surface is a square with opposite sides identified. This is a torus. If we let $\bar{v}$ denote one of the sides of the square then $T_{\theta}$ is rotation by $\cot(\theta)$ mod 1 (or $2\pi \cot(\theta)$ on the unit circle).
\begin{Bob} A line segment in $Q$ is called a \emph{saddle connection} if it connects two vertices of the surface and has no vertex in its interior.
\end{Bob}

%


To state the results of this paper we introduce two terms, motivated by Kurzweil's Theorem, in the setting of $\mathbb{Z}$ and $\mathbb{R}$ actions.

\begin{Bob}
Let $(X,d)$ be a compact metric space, let $\mathcal{F}$ be a family of $\mu$-measure preserving
$\mathbb{Z}$-actions ${T \colon (X,d) \to (X,d)}$  and let $\nu$ be
a measure on $\mathcal{F}$. We say $\mathcal{F}$ has the \emph{Kurzweil
property} if given a decreasing sequence $\{a_i\}_{i=1}^{\infty}$
such that $\underset{i=1}{\overset{\infty}{\sum}}
\mu\left(B(T^ix,a_i)\right)$ diverges for all $x$, $\nu$-almost
every $T \in \mathcal{F}$ satisfies the property that
\begin{center}
$\LS B(T^ix,a_i)$ has full $\mu$-measure for $\mu$-almost every $x$.
\end{center}

 $\mathcal{F}$ has the \emph{strong Kurzweil property} if given a decreasing sequence $\{a_i\}_{i=1}^{\infty}$
  such that $\underset{i=1}{\overset{\infty}{\sum}} \mu\left(B(x,a_i)\right)$
   diverges for all $x$ then $\nu$-almost every
   $T \in \mathcal{F}$ satisfies the property that
\begin{center}
$\LS T^{-i}\left(B(y,a_i)\right)$ has full $\mu$-measure for every
$y$.
\end{center}
\end{Bob}
In Kurzweil's Theorem $\mathcal{F}=\{R_{\alpha}\}_{\alpha \in [0,1)}$ and $\nu=\Leb$.

We introduce the terminology because, like rotations the typical IET does not satisfy MSTP.
However, as a family they satisfy shrinking target properties. The Kurzweil property addresses $B(T^ix,a_i)$ and is dual to MSTP which addresses $T^{-i}\left(B(y,a_i)\right)$.
The strong Kurzweil property is motivated by rephrasing Kurzweil's result to be closer to the MSTP. By Fubini's Theorem the
 strong Kurzweil property implies the Kurzweil property because
$$\{(x,y): y \in \LS B(T^ix,a_i)\}=\{(x,y):x \in \LS T^{-i}\left(B(y,a_i)\right)\}.$$
 In the case of rotations the Kurzweil and strong Kurzweil properties are equivalent because showing that there exists $x_0$ such that
$\Leb\left(\LS B(R^i_{\alpha}x_0,a_i)\right)=1$
 is equivalent to the fact that for every $x$ we have
$\Leb\left(\LS B(R^i_{\alpha}x,a_i)\right)=1$
 which is equivalent to the fact that for every $y$ we have
$\Leb\left(\LS R_{\alpha}^{-i}(B(y,a_i))\right)=1$.

The Kurzweil property for $\Leb$ preserving actions on $[0,1)$ considers only decreasing sequences with divergent sum. These sequences are called \emph{standard}. 
We state a few properties.
\begin{enumerate}
\item Let $r \in \mathbb{N}$ and $r\geq 2$. Define $\{b_i\}_{i=1}^{\infty}$ by $b_i= a_{r^k}$ for $r^{k-1}\leq i <r^k$. If $\{a_i\}_{i=1}^{\infty}$ is standard then $\{b_i\}_{i=1}^{\infty}$ is standard. It is obviously non-increasing. To see that it has divergent sum notice that by our assumption that $\{a_i\}_{i=1}^{\infty}$ is non-increasing $\underset{k=0}{\overset{\infty}{\sum}}r^{k+1}a_{r^k}\geq \underset{i=1}{\overset{\infty}{\sum}}a_i$. Also notice that $\underset{k=0}{\overset{\infty}{\sum}}r^{k+1}a_{r^k}\leq r^2  \underset{i=1}{\overset{\infty}{\sum}}b_i.$
\item If $\{a_i\}_{i=1}^{\infty}$ is standard and $S$ is a subset of $\mathbb{N}$ with positive lower density then $\underset{i \in S}{\sum} a_i= \infty.$  To see this notice that if $S$ has positive lower density the there exists $c>0$ such that for all big enough $r$ and $k$ we have $|S\cap [r^k,r^{k+1}]|>cr^k$. Therefore $ \underset{i \in S}{\sum} a_i$ diverges with $\underset{k=1}{\overset{\infty}{\sum}} cr^{k}a_{r^{k+1}}$.
\item If $b_i\leq a_i$ then $\LS B(T^ix,b_i)\subset \LS B(T^ix,a_i)$.
\item To establish the Kurzweil and strong Kurzweil properties it suffices to consider $\{a_i\}_{i=1}^{\infty}$ with $\underset{n \to \infty } {\limsup} \, na_n=0$. This follows from the previous property.
\end{enumerate}

We now extend the definition of the Kurzweil and strong Kurzweil
properties to $\mathbb{R}$-actions.
\begin{Bob}
Let $\mathcal{F}$ be a family of $\mu$-measure preserving
$\mathbb{R}$-actions ${F \colon (X,d) \to (X,d)}$ and let $\nu$ be a
measure on $\mathcal{F}$. $\mathcal{F}$ is said to satisfy the
\emph{ Kurzweil property} if for any decreasing function $f \colon
\mathbb{R} \to \mathbb{R}^+$ such that $\int_0 ^{\infty}
\mu\left(B(F^tx,f(t))\right)dt=\infty$ for all $x$, $\nu$-almost
every $F \in \mathcal{F}$ satisfies the property that
\begin{center} $\CLS B(F^tx,f(t))$ has full $\mu$-measure for $\mu$-almost every $x$.
\end{center}

 $\mathcal{F}$ is said to satisfy the \emph{strong Kurzweil property} if for any decreasing function
  $f \colon \mathbb{R} \to \mathbb{R}^+$ such that
  $\int_0 ^{\infty} \mu\left(B(x,f(t))\right)dt=\infty$ for all $x$ then $\nu$-almost every
   $F \in \mathcal{F}$ satisfies the property that
\begin{center} $\CLS F^{-t}\left(B(y,f(t))\right)$ has full $\mu$-measure for every $y$.
\end{center}
\end{Bob}

\begin{thm}\label{main theorem} Let $Q$ be a translation surface then $$\mathcal{F}=\{F_{\theta}^t:Q \to Q \text{ flow in direction } \theta \text{ with unit speed}\}$$ and measure $\LebP$ satisfies the strong Kurzweil property.
\end{thm}
Theorem \ref{main theorem} holds for every translation surface and therefore applies to the billiard in any fixed rational polygon. This is because following \cite{KZ} we may unfold a billiard table that is a rational polygon to reinterpret the billiard trajectories as straight line flows on a translation surface.
\begin{rem}\label{set} Theorem \ref{main theorem} says  that if $$S_{\theta}(f)=\left\{(x,y) \in Q \times Q: x \in \CLS F_{\theta}^{-t}\left(B(y,f(t))\right)\right\}$$ then for any fixed $f$ decreasing with divergent integral almost every $\theta$ has the property that
$$\LebF\left(\{x: (x,y_0) \in S_{\theta}(f)\}\right)=1 \text{ for every }y_0 \in Q.$$ Section 3 proves this and also Proposition \ref{other} which shows that,
 $$\LebF\left(\{y:(x_0,y) \in S_{\theta}(f) \}\right)=1  \text{ for every } x_0 \in Q.$$
It is easy to see that $S_{\theta}(f)$ is measurable.
\end{rem}
\begin{cor} \label{strong kurz for IET} Interval exchange transformations with irreducible permutations and measure $\bold{m}_d$ satisfy the strong Kurzweil property.
\end{cor}
Establishing the above corollary and Fubini's Theorem would not establish Theorem \ref{main theorem}, which holds for every translation surface.
\begin{rem} Corollary \ref{strong kurz for IET} strengthens Theorem \ref{AU}. Given any standard $\{a_i\}_{i=1}^{\infty}$, almost every IET has the property that ${\Leb\left(\LS T^{-i}(B(y,a_i))\right)=1}$  simultaneously for all $y$.
\end{rem}
These results state that IETs satisfy strong shrinking target properties, however this is not the complete picture.
\begin{thm} \label{rigidthm} For almost every IET $T$, there exists a standard sequence ${\bold{a}_T:=\{a_i\}_{i=1}^{\infty}}$ such that $$\Leb\left(\LS B(T^ix,a_i)\right)=0$$ for $\Leb$-almost every $x$.
\end{thm}
That is, almost every IET does not satisfy MSTP.
This result is a little deceptive because
\begin{thm}\label{full measure thm}  There exists a full measure set of IETs $\mathcal{V}$ such that for any standard sequence $\bold{a}$  where $ia_i$ is eventually monotone $$\Leb\left(\LS B(T^ix,a_i) \right)=1$$ for any $T \in \mathcal{V}$ and for every $x$.
\end{thm}
The condition on sequences in this theorem is common and appears, for example, in Theorem \ref{march} and earlier in \cite[Theorem 32]{Khinchin}. 
 One way to think of Theorem \ref{full measure thm} is that it says that for almost every IET the standard sequences such that $\Leb\left(\LS B\left(T^ix,a_i\right)\right)=0$ for some $x$ violate a mild regularity condition.
\begin{rem} For rotations there is a necessary and sufficient condition: $\Leb\left(\LS B(R_{\alpha}^ix,a_i)\right)=1$ for every $x$ and any standard sequence $\{a_i\}_{i=1}^{\infty}$ where $ia_i$ is eventually monotone if and only if $\underset{n \to \infty}{\limsup} \frac {\log (q_n(\alpha))}{n}< \infty$  where $q_n(\alpha)$ is the denominator of the $n^{th}$ convergent of $\alpha$. This set excludes all Louiville $\alpha$, however it also excludes some $\alpha$ that are of Roth type. Recall that $\alpha$ is said to be of Roth type if for any $\epsilon>0$ we have $\min\{|R_{\alpha}^n(0)-0|, |R_{\alpha}^n(0)-1|\}>n^{-1-\epsilon}$ for all but finitely many $n$. Almost every $\alpha$ is of Roth type. The proof of the sufficiency of this condition is a straightforward application of Section \ref{full sec}. To show the necessity assume $\alpha $ does not have this form and use the target $a_j=\frac{1}{j\log(N_i)}$ for all $N_{i-1}\leq j<N_i$ for a sequence of  $N_i$ chosen similarly to Section \ref{rigid}. 
\end{rem}

The plan for this paper is to first establish the Kurzweil property for IETs and flows on flat surfaces. Then we establish the strong Kurzweil property (Theorem \ref{main theorem}). Then we show that almost every IET fails MSTP (Theorem \ref{rigidthm}), which is a straightforward application of Veech's proof that almost every IET is rigid. Then we use Rauzy-Veech induction to show Theorem \ref{full measure thm}.

\section{Proof of the Kurzweil property}
The main results of this section, Proposition \ref{kurz for iet} and Corollary \ref{kurz for flow}, establish the Kurzweil property for flat surfaces. The proof of the strong Kurzweil property in the next section is a little more complicated but mainly follows the lines of this proof.

\begin{thm} \label{kms}(Kerkchoff, Masur and Smillie \cite{KMS}) Let $Q$ be a translation surface then for almost every $\theta$, $F_{\theta}^t$ and $T_{\theta}$ are uniquely ergodic with respect to $\LebF$ and $\Leb$ respectively.
\end{thm}
This allows us to make the following reduction which is a straightforward application of ergodicity whose proof is included for completeness.
\begin{prop} \label{LS reduction} If $T$ is a $\Leb$-ergodic IET, $\{a_i\}_{i=1}^{\infty}$ is non-increasing and  $x$ has $${\lambda \left(\underset{n=1}{\overset{\infty}{\cap}} \underset{i=n}{\overset{\infty}{\cup}} B(T^i x, a_i)\right)>0} \text{ then }\Leb\left(\LS B(T^ix,a_i)\right)=1.$$ If a positive measure set of $x$ have $${\Leb\left(\LS B(T^ix,a_i)\right)=1} \text{ then }{\Leb\left(\LS B(T^ix,a_i)\right)=1}$$ for $\Leb$-almost every $x$. If $$\Leb\left(\LS T^{-i}(B(y,a_i))\right)>0 \text{ then }\Leb\left(\LS T^{-i}(B(y,a_i))\right)=1.$$ If a positive measure set of $y$ have $${\Leb\left(\LS T^{-i}(B(y,a_i))\right)=1}\text{ then }{\Leb\left(\LS T^{-i}(B(y,a_i))\right)=1}$$
for $\Leb$-almost every $y$.\end{prop}
\begin{proof} Consider the measurable set $G=\{(x,y): y \in \LS B(T^ix,a_i)\}$. Because $\bold{a}$ is non-increasing if $(x,y) \in G$ then $(T(x),y) \in G$. Also, $T$ is a piecewise isometry, so for any $y$ outside the orbits of discontinuity points, $(x,y)  \in G$ implies $(x,T^{-1}(y)) \in G$. Therefore, by ergodicity if $$\Leb\left(\{y:(x_0,y) \in G\}\right) >0 \text{ then }\Leb\left(\{y: (x_0,y) \in G \}\right)=1.$$ Also $$\Leb\left( \{x:(x,y_0) \in G\}\right)>0 \text{ implies }\Leb\left( \{x:(x,y_0) \in G\}\right)=1.$$\end{proof}
This implies that it suffices to show that for almost every $\theta$ (those such that $F_{\theta}^t$ and $T_{\theta}$ are uniquely ergodic), $\Leb (\LS B(T_{\theta}^ix,a_i))>0$. To establish this property we use the following result.
\begin{lem} \label{separated} Let $\epsilon>0,e>0$ and $n,t \in \mathbb{N}$. If $\{z_1,...,z_n\}\subset \mathbb{R}$ are $\frac e n$ separated and $S\subset \mathbb{R}$ is a set of measure $\epsilon$ that is the union of $t$ intervals  then  the inequality $$\Leb \left(\underset{i=1}{\overset{n}{\cup}}B(z_i,\delta) \backslash S \right) > (n-2t-\frac{n \epsilon}{e} )\delta $$ holds for any $\delta< \frac e {2n}$.
\end{lem}
\begin{proof} At most $\frac{\epsilon}{e}+2t$ of the points can lie within a $\frac e {2n}$ neighborhood of $S$. This is because an interval of length $l$ can contain at most $\lceil \frac l e \rceil$ points that are $e$ separated. Therefore all but $\frac{\epsilon}{e}+2t$ of the points $\{z_1,z_2,...,z_n\}$ have $B(z_i, \delta) \cap S=\emptyset$ and the lemma follows.
\end{proof}
\begin{rem}\label{meaning of sep} The statement of this lemma is abstract. It is used in this paper to show that under suitable assumptions for some $\epsilon>0$ we have $\Leb\left(\underset{i=N}{\overset{\infty}{\cup}}B(T^ix, a_i)\right)>\epsilon$ for any $N$. This is done by showing that
 $$\Leb\left(\underset{i=r^k}{\overset{r^{k+1}}{\cup}}B(T^ix, a_i)\backslash \underset{i=N}{\overset{r^k}{\cup}}B(T^ix, a_i)\right)$$ is proportional to $r^ka_{r^{k+1}}$ when $ \Leb \left(\underset{i=N}{\overset{r^k}{\cup}}B(T^ix, a_i) \right)$ is small and $\{T^i(x)\}_{i=r^k}^{r^{k+1}}$ has many $\frac{e}{r^{k+1}}$ separated points.
\end{rem}
Motivated by this lemma we will assume $\underset{n \to \infty}{\lim} \, na_n=0$ and make the following definition to determine when we can apply Lemma \ref{separated}.
\begin{Bob} Let $e_T(n)$ be the smallest distance between discontinuities of $T^n$.
\end{Bob}
\begin{thm}(Boshernitzan) Let $Q$ be a polygon with quadratic growth of saddle connections, then $$\underset{\epsilon \to 0}{\lim} \quad \underset{n>0}{\sup} \quad \Leb\left(\{\theta: e_{T_{\theta}}(n)< \frac {\epsilon}{n}\}\right)=0.$$
\end{thm}
This appears in \cite[page 750]{dukeJ}. Recall that a translation surface is said to have quadratic growth of saddle connections if there exists a constant $C$ such that the number of saddle connections with length less than $L$ is smaller than $CL^2$ for all large $L$.
By repeating the arguments in the proof of this result in a translation surface one obtains the following corollary. The key idea in the proof is the fact that that $e_{T_{\theta}}(n)< \frac {\epsilon}{n}$ implies that $\theta$ is at most $O(\frac {\epsilon} {n^2})$ away from the direction of a saddle connection that crosses the transversal at most $n$ times.
\begin{cor}\label{discont2}  Let $Q$ be any translation surface with at most quadratic growth of saddle connections, then $$\underset{\epsilon \to 0}{\lim} \quad \underset{n>0}{\sup} \quad \Leb \left(\{\theta: e_{T_{\theta}}(n)< \frac {\epsilon}{n}\}\right)=0.$$
\end{cor}
\begin{thm}(Masur) Flat surfaces have at most quadratic growth of saddle connections.
\end{thm}
This was proven in \cite{sad con}. For an effective version proven by elementary methods see \cite{vor sad con}.
\begin{prop} \label{interval} (Boshernitzan \cite[Lemma 4.4]{rank2}) If the orbits of the discontinuities of $T$ are infinite and distinct then for any interval $J$ of size $e(n)$ there exist natural numbers ${p \leq 0 \leq q}$ such that
\begin{enumerate}
\item $q-p \geq n$
\item $T^i$ acts continuously on $J$ for $p \leq i <q$
\item $T^i(J) \cap T^j(J)= \emptyset$ for $p \leq i < j <q$.
\end{enumerate}
\end{prop}
\begin{cor} \label{key} For each translation surface and $\epsilon>0$ there exists $c_{\epsilon}<1$ 
such that for each $n$ a set of $\theta$ of measure $1-c_{\epsilon}$ has at least half of the points in $\{T_{\theta}^{n}(x),T_{\theta}^{n+1}(x),...,T_{\theta}^{2n}(x)\}$ pairwise $\frac{ \epsilon }{r^{k+1}}$ separated.
\end{cor}
\begin{proof} Choose $c_\epsilon<  \LebP\left(\{\theta: e_{T_{\theta}}<\frac \epsilon{2n}\}\right).$ If $e_{T_{\theta}}>\frac {2\epsilon} n$ and $\{B(x,\frac{\epsilon}{n}),...,B(T^{\frac{n}{2}}x,\frac{\epsilon}{n})\}$ are not disjoint or $T^i$ is not continuous on $B(x,\frac{\epsilon}{n})$ for some $0 \leq i \leq \frac{n}{2}$  then by Proposition \ref{interval} we have $\{B(T^{\frac{n}{2}}x,\frac{\epsilon}{n}),...,B(T^nx,\frac{\epsilon}{n})\}$ are disjoint and ${B(T^{i+\frac{n}{2}}x, \frac{\epsilon}{n})=T^iB(T^{\frac{n}{2}}x,\frac{\epsilon}{n})}$ for any $0 \leq i\leq \frac n 2$.
\end{proof}
To establish the Kurzweil property we show that for every $\delta>0$ there exists an $\epsilon_2>0$ such that for any set of directions $\mathcal{V}$ with $\LebP(\mathcal{V})>\delta$ there exists $\mathcal{U} \subset \mathcal{V}$ with $\LebP\left(\mathcal{U}\right)>0$ and  $\Leb\left(\LS B(T_{\theta}^ix,a_i)\right)>\epsilon_2$ for every $\theta \in \mathcal{U}$.  To establish this Corollary \ref{key} and Lemma \ref{separated} are used.
By Theorem \ref{kms} and Proposition \ref{LS reduction} this implies $\Leb\left(\LS B(T_{\theta}^ix,a_i)\right)=1$ for almost every $\theta.$
\begin{prop}\label{kurz for iet} Let $Q$ be a translation surface, then $\{T_{\theta}\}_{\theta\in(0,1)}$ with measure $\LebP$ satisfies the Kurzweil property.
\end{prop}
\begin{proof}
Assume not. Then there exists a standard sequence $\{a_i\}_{i=1}^{\infty}$ and a set of directions $\mathcal{V}$, such that $\LebP(\mathcal{V})>2\delta$ and for any $\theta \in \mathcal{V}$ we have ${\Leb\left(\LS B(T_{\theta}^ix,a_i)\right)=0}$ for almost every $x$ (this follows by Theorem \ref{kms} and Proposition \ref{LS reduction}).  Fix $r>8$ and choose $\epsilon>0$ such that ${\LebP\left(\{\theta: e_{T_{\theta}}(n)< \frac{\epsilon}{n}\}\right)<\frac{\delta} 2}$ for all $n$.
 There exists $N$, $\mathcal{V}'$ such that $\LebP(\mathcal{V}')>\delta$ and
  $$\Leb \left( \underset{i=N}{\overset{\infty}{\cup}} B(T_{\theta}^i x, a_i)\right)
  < \epsilon \frac 1 r \text{ for any }\theta \in \mathcal{V}'.\text{ (See Remark \ref{meaning of sep}.)} $$  Denote $\frac {\epsilon}{r}$ by $\epsilon_2$.
  To see that $\mathcal{V}'$ exists notice that if
  ${\Leb\left(\LS B(T^i_{\theta}x,a_i)\right)=0}$ then for any
  ${\epsilon>0}$ there exists $N$ (depending on $\epsilon$ and $T_{\theta}$) such that
  ${\Leb
  \left(\underset{i=N}{\overset{\infty}{\cup}}B(T_{\theta}^ix,a_i)\right)<\epsilon}$.
   Then appeal to the countable subadditivity of measures.

If $\theta \in \mathcal{V}'$ and $e_{T_{\theta}}(n)>
\frac{\epsilon}{n} $ (which is the case for a set of measure at
least $\frac{\delta}{2}$ because $\LebP(\mathcal{V}')>\delta$ and
$\LebP\left(\{\theta: e_{T_{\theta}}(n)<\frac
{\epsilon}{n}\}\right)<\frac {\delta}{2}$) then by Corollary
\ref{key} and Lemma \ref{separated}
$$\Leb\left(\underset{i=r^{k-1}}{\overset{r^k}{\cup}}B(T_{\theta}^ix,a_{r^k})
\backslash
\underset{i=N}{\overset{r^{k-1}}{\cup}}B(T_{\theta}^ix,a_i)\right)>
a_{r^k}(\frac 1 2 (r^k-r^{k-1}) -2r^{k-1}-\frac
{\epsilon_2}{\epsilon}r^{k-1}).$$
 By our assumptions on $r, \epsilon, \epsilon_2$ this is greater than $a_{r^k}(\frac 1 2 r^k-\frac 7 2 r^{k-1}).$

From this it follows that
$$\int_{\mathcal{V}'}\Leb\left(\underset{i=r^{k-1}}{\overset{r^k}{\cup}}B(T_{\theta}^ix,a_{r^k})
 \backslash \underset{i=N}{\overset{r^{k-1}}{\cup}}B(T_{\theta}^ix,a_i)\right) d\theta> \frac 1 2  \LebP(\mathcal{V}') a_{r^k}(\frac 1 2 r^k-\frac 7 2 r^{k-1}).$$ With the observation that $\underset{k=1}{\overset{\infty}{\sum}} r^k a_{r^k}$ diverges, we derive a contradiction. This is because iterating the above argument for an increasing sequence of $k$ shows that there must be $\theta \in \mathcal{V}'$ with $\Leb\left(\underset{i=N}{\overset{\infty}{\cup}}B(T^i_{\theta}x,a_i)\right)>\epsilon$, which contradicts the definition of $\mathcal{V}'$.
\end{proof}
By Fubini's Theorem we get the following result.
\begin{cor} The set of IETs with irreducible permutations and Lebesgue measure on the parameterizing space satisfies the Kurzweil property.
\end{cor}
\begin{cor}\label{kurz for flow} Let $Q$ be a translation surface, then $\{F_{\theta}^t\}_{\theta \in (0,1)}$ with measure $\LebF$ satisfies the Kurzweil property.
\end{cor}
\begin{proof} Consider the full measure set of directions such that all points outside of the orbit of a singularity have a unique pre-image on the transversal. Pick one such direction $\theta$ and let $k$ be the greatest first return time of $F_{\theta}^t$ to the transversal. If $u,v$ are points in $Q$ and $x_u$ and $x_v$ are the pre-images of $u$ and $v$ on the transversal under $F_{\theta}^t$ then $u \in \CLS B\left(F_{\theta}^tv,f(t)\right)$ whenever $$x_u \in \LS B\left(T_{\theta}^i\left(x_v\right),f\left(k\left(i+1\right)\right)\right).$$ With the observation that $a_i=f(ki)$ is a standard sequence the result follows from Proposition \ref{kurz for iet}.
\end{proof}

\section{Strong Kurzweil property}

This section establishes the strong Kurzweil property by first showing a slightly different dual property. Throughout this section we assume that we are in a fixed translation surface $Q$.
\begin{prop}\label{measurable} For any sequence $\{a_i\}_{i=1}^{\infty}$ the set
 $$\left\{ \theta: \exists x \in [0,1)\text{ with } \Leb\left(\LS B(T_{\theta}^ix,a_i)\right)=0 \right\}$$ is measurable.
\end{prop}
It suffices to show the fact that $$U_{N,M,\epsilon}:=\{\theta: \exists x \in [0,1) \text{ with } \Leb\left(\underset{i=N}{\overset{M}{\cup}}B(T_{\theta}^ix,a_i)\right)<\epsilon \}$$ is measurable for all $N, M , \epsilon$ which follows from the next lemma.

\begin{lem}\label{reduction} For any sequence $\{a_i\}_{i=1}^{\infty}$, $\delta>0$ and $M \in \mathbb{N}$ there exists $A_1,A_2,$..., a countable partition of $\LebP$-almost
 all of $[0,1)$ into intervals,
 and associated points $\{x_j\}_{j=1}^{\infty}$ such that for each $\theta \in A_j$ we have $${\Leb\left(\underset{i=1}{\overset{M}{\cup}}B(T_{\theta}^ix_j, a_i)\right)< \underset{x \in [0,1)}{\inf}\Leb \left(\underset{i=1}{\overset{M}{\cup}}B(T_{\theta}^ix, a_i)\right)+\delta}.$$
\end{lem}
To be explicit, the $A_i$ are sets of directions that parametrize the IETs and the $x_i$ are points that the IETs act on.
\begin{proof}
For fixed $x$ and $M$ let $\{\phi_i\}_{i=1}^{\infty}$ be the sets of directions of saddle connections which intersect the transversal $M$ or fewer times. Let $\{\phi'_i\}_{i=1}^{\infty}$ be the set of directions $\phi_i'$ such that $T_{\phi_i'}^j x$ is a discontinuity of $T_{\phi_i'}^M$ for some $j<M$. This may be a finite set. These two sets can have at most two accumulation points, the direction of the transversal. On any interval which does not contain one of these points $T_{\theta}^i(x)$ changes continuously with $\theta$. Therefore $\Leb\left(\underset{i=1}{\overset{M}{\cup}}B(T_{\theta}^ix_j, a_i)\right)$ changes continuously with $\theta$ on these intervals and the lemma follows.
\end{proof}
This establishes that $A_i \cap U_{N,M,\epsilon}$ is contained in an open set union at most 2 points which in turn is contained in $A_i \cap U_{N,M,\epsilon+\delta}$. Thus for any $\delta>0$ there is a measurable set contained in $U_{N,M,\epsilon+\delta}$ which contains $U_{N,M,\epsilon}$. Intersecting these measurable sets shows that $U_{N,M,\epsilon}$ is measurable.

We now establish a closely related property that is easier to show than the strong Kurzweil property and is neither stronger nor weaker. See Remark \ref{set}.
\begin{prop}\label{other} Let $\{a_i\}_{i=1}^{\infty}$ be a standard sequence. $$\LebP\left(\{\theta: \exists x \in [0,1)\text{ with } \Leb\left(\LS B(T_{\theta}^ix,a_i)\right)=0 \}\right)=0.$$
\end{prop}
\begin{proof} Assume not.
By Proposition \ref{LS reduction} we may assume that there exists a set of directions $\mathcal{V}$ with $\LebP(\mathcal{V})>2\delta$ and for every
$\theta \in \mathcal{V}$ there exists $x_{\theta}$
 such that $ \Leb \left( \LS B(T_{\theta}^ix_{\theta},a_i) \right)=0.$
  Let $r>8$ be a natural number and choose $\epsilon>0$ such that
${\LebP\left(\{\theta: e_{T_{\theta}}(n)<\frac{\epsilon}{n}\}\right)<\frac{\delta} 2}$.
Choose $\epsilon_2$ such that
$\epsilon_2<\frac {\epsilon} {r}$ and $\epsilon_2<\frac 1 {8r}$.
Choose $r^N$ such that there exist
$\mathcal{V}' \subset \mathcal{V}$ with $\LebP(\mathcal{V}')> \delta$
 and each $\theta \in \mathcal{V}'$ satisfies
$\Leb\left(\underset{i=r^N}{\overset{\infty}{\cup}}B(T^ix,a_i)\right)< \epsilon_2$.
 Choose $r^M$ such that
$\frac 1 {12} \underset{i=r^N}{\overset{r^M}{\sum}} r^ia_{r^i}> \frac 1 r$.
As was seen in the proof of Proposition \ref{kurz for iet} if $\theta \in \mathcal{V}'$
 and $e_{T_{\theta}}(n)> \frac{\epsilon}{n} $ then
 $$\Leb\left(\underset{i=r^{k-1}}{\overset{r^k}{\cup}}B(T^ix,a_{r^k}) \backslash \underset{i=r^N}{\overset{r^{k-1}}{\cup}}B(T^ix,a_i)\right)> a_{r^k}\left(\frac 1 2 (r^k-r^{k-1}) -2r^{k-1}-\frac {\epsilon_2}{\epsilon}r^{k-1}\right).$$
 By our assumptions on $r, \epsilon, \epsilon_2$ this is greater than $a_{r^k}\left(\frac 1 2 r^k-\frac 7 2 r^{k-1}\right).$
Following Lemma \ref{reduction} choose a partition of $\mathcal{V}$ into measurable sets $A_1,A_2,...$ such that for each $A_j$ there is $x_j$ with the property that for each $T \in A_j$
$$\Leb\left(\underset{i=r^N}{\overset{r^M}{\cup}} B(T^i(x_j),a_i)\right)<\underset{ x \in [0,1)}{\inf}\Leb \left(\underset{i=r^N}{\overset{r^M}{\cup}}B(T^ix,a_i)\right)+ \frac 1 {4r} .$$
 Notice that under our assumptions, which imply that
$\epsilon_2 + \frac 1 {4r}< \frac 1 {24}$,
 $$\underset{j=1}{\overset{\infty}{\sum}} \int_{A_j}\Leb\left(\underset{i=r^N}{\overset{r^M}{\cup}}B(T^i(x_j),a_i)\right)dT>\frac 1 2 \LebP(\mathcal{V}') \frac 1 {12} \underset{i=r^N}{\overset{r^M}{\sum}} r^ka_{r^k}  >(\epsilon_2 +\frac 1 {4r})\LebP(\mathcal{V}').$$
 This derives a contradiction to the definition of $\mathcal{V}'$.
\end{proof}
\begin{rem}
In the proof we used Lemma \ref{reduction} to avoid any possibility of measurability concerns with the integral; naively one would want to take $\int_{\mathcal{V}'}\Leb\left( \underset{i=r^N}{\overset{\infty}{\cup}} B(T_{\theta}^i(x_{\theta}),a_i)\right)d\theta$.
\end{rem}
Fubini's Theorem and Proposition \ref{other} show that for every standard sequence
 $\{a_i\}_{i=1}^{\infty}$, $\LebP$-almost every $\theta$ and $\Leb$-almost every $y$
 we have $y \in \LS B(T_{\theta}^ix,a_i)$ for almost every $x$.
 Strengthening this to show that for $\LebP$-almost every $\theta$ and every $y$,
 $$\left\{x: y \in \LS B(T_{\theta}^ix,a_i)\right\}$$
has full measure establishes the strong Kurzweil property. The first step is
\begin{lem} The set $$\left\{\theta: \exists y\in [0,1) \text{ with } \Leb\left(\LS T^{-i}_{\theta}B(y,a_i)\right)<1\right\}$$ is measurable.
\end{lem} This is identical to Proposition \ref{measurable}.

Step two is establishing an analogue of Lemma \ref{separated} for this situation.
\begin{lem} \label{separated 2}Let $e, \epsilon>0$ and $r, k \in \mathbb{N}$. If $e_{T}(r^{k+1})>\frac e {r^{k+1}}$ and $S$ is the union of at most $r^{k}$ balls and  $\Leb(S)>\epsilon$ then $$\Leb\left(\underset{i=r^{k}}{\overset{r^{k+1}}{\cup}} T^{-i}B(y,\delta)\backslash S\right)> \frac 1 4 (\frac 1 2 (r^{k+1}-r^{k})-2r^{k}- \frac {\epsilon}{e} r^{k+1})\delta$$ provided that $\delta< \frac e {2r^{k+1}}$.
\end{lem}
\begin{proof}
Fix $y$ and consider $$\left\{T^{-r^k}(B(y,\delta)),..., T^{-r^{k+1}}(B(y,\delta))\right\}.$$ By Proposition~ \ref{interval}, we have that if
$e_T(r^{k+1}) > 2\delta$ then each $T^{-i}(B(y,\delta))$ in this set is the union of at most 2 intervals. Moreover, either
$T^{-i}(B(y,\frac e {r^{k+1}}))$ is connected for ${0<i<\frac {r^{k+1}-r^k}{2}}$ or it isn't.
In the first case we have $\frac {r^{k+1}-r^k}{2}$ inverse images of $B(y,\delta)$
at least $\frac e {r^{k+1}}$ separated  and consider $\underset{i=r^k}{\overset{\frac 1 2 (r^{k+1}-r^k)}{\cup}}T^{-i}(B(y,\delta))$.
 In the other case if $B(y,\delta)$ is split then the larger of the two pieces have
 $\frac {r^{k+1}-r^k}{2}$ inverse images that are $\frac e {r^{k+1}}$ separated from each other and if $B(y,\delta)$
 does not split we have $\frac {r^{k+1}-r^k}{2}$ copies of $B(y,\delta)$ that are  $\frac e {r^{k+1}}$ separated from each other.
It follows that if
$$\Leb\left(\underset{i=N}{\overset{r^{k-1}}{\cup}}T^{-i}(B(y,a_i))\right)<\epsilon_2 \text{ and }e_{T}(r^k)>\frac{\epsilon}{r^k}$$
then
\begin{multline}\lefteqn{\Leb\left(\underset{i=r^{k-1}}{\overset{r^{k}}{\bigcup}}T^{-i}(B(y,a_i))
\backslash
(\underset{i=N}{\overset{r^{k-1}}{\bigcup}}T^{-i}(B(y,a_i)))\right)}\\>
\frac 1 2 \left(\frac 1 2 (r^k -r^{k-1})-2r^{k-1}-\frac
{\epsilon_2}{\epsilon}r^k\right)\frac 1 2 a_{r^k}.\end{multline}
\end{proof} Proceeding analogously to Proposition \ref{other} we obtain,
\begin{prop}\label{skurz} For any translation surface $Q$ the set $\{T_{\theta}\}_{\theta \in (0,1)}$ with Lebesgue measure $\LebP$ satisfies the strong Kurzweil property.
\end{prop}
\begin{proof}[Proof of Theorem \ref{main theorem}] This follows from Proposition \ref{skurz} by a parallel argument to how Corollary \ref{kurz for flow} follows from Proposition \ref{kurz for iet}.
\end{proof}
By Fubini's theorem we obtain Corollary \ref{strong kurz for IET}.

\section{Almost every IET fails MSTP}\label{rigid}
Analogously to Kurzweil's result, almost every IET does not satisfy MSTP. To prove Theorem \ref{rigidthm} we recall a theorem, which shows that almost every IET is rank 1 and rigid:
\begin{thm} (Veech \cite[Theorem 1.4 Part I]{metric}) For almost every IET $T$ given any $\epsilon>0$ there exists $N \in \mathbb{N},$ and an interval $J \subset [0,1)$ such that:
\begin{enumerate}
\item $J \cap T^n(J)=\emptyset$ for  $0<n < N$.
\item $T$ is continuous on $T^n(J) $ for $ 0 \leq n<N$.
\item $\Leb\left(\underset{n=1}{\overset{N}{\cup}}T^n(J)\right)>1-\epsilon$.
\item $\Leb\left(T^N(J) \cap J\right)>(1-\epsilon) \Leb(J)$.
\end{enumerate}
\end{thm}
Let $T$ be an IET such that the above Theorem holds. Let $\epsilon_i=\frac 1 {3^i}$. Choose $N_{i} \in \mathbb{N}$ increasing and intervals $J_i \subset [0,1)$ such that:
\begin{enumerate}
\item $J_i \cap T^n(J_i)=\emptyset$ for  $0<n < N_i$.
\item $T$ is continuous on $T^n(J_i) $ for $ 0 \leq n<N_i$.
\item $\Leb\left(\underset{n=1}{\overset{N_i}{\cup}}T^n(J_i)\right)>1-3^{-i}$.
\item $\Leb\left(T^{N_i}(J_i) \cap J_i\right)> (1-3^{-i}) \Leb(J_i)$.
\end{enumerate}
Notice that $$|T^{N_j}x-x|< \frac 1 {N_j 3^j} \text{ for any
} x \in \underset{n=1}{\overset{N_j}{\cup}}T^n\left(J_j \cap T^{-N_j}(J_j)\right).$$
 This is a set of measure at least $1-2(3^{-j} )$. Likewise,
$$|T^{kN_j}x-x|<\frac k {N^j 3^j} \text{ for } {x \in \underset{n=1}{\overset{N_j}{\cup}}T^n\left(J_j \cap T^{-N_j}(J_j) \cap ... \cap T^{-kN_j}(J_j)\right)}.$$
 This set has measure at least ${1-(k+1)3^{-j}}$.
Let $a_i= \frac 1 {2^j N_j}$ for all ${2^{j-1}N_{j-1} \leq i < 2^jN_j}$. If
$$x \in \underset{n=1}{\overset{N_j}{\cup}}T^n\left(J_j \cap T^{-N_j}(J_j) \cap ... \cap T^{-2^jN_j}(J_j)\right)$$ then
$$\Leb\left(\underset{i=2^{j-1}N_{j-1}}{\overset{2^jN^j}{\cup}}B(T^ix,a_i)\right)<N^j\frac 1 {2^jN_j}+2^jN_j \frac 1 {3^jN_j}.$$
 Almost every $x$ is eventually in
$$\underset{n=1}{\overset{N_j}{\cup}}T^n\left(J_j \cap T^{-N_j}(J_j) \cap ... \cap T^{-2^jN_j}(J_j)\right) $$
 for all large enough $j$
(because $\underset{j=1}{\overset{\infty}{\sum}} (2^j+1)3^{-j}< \infty$). Therefore ${\Leb \left(\LS B(T^ix,a_i)\right)=0}$ for almost all $x$.
In fact, by examining how $x$ travels in $\underset{n=1}{\overset{N_j}{\cup}}T^n(J_j)$
 one gets $\Leb \left(\LS B(T^ix,a_i)\right)=0$ for every $x$.
Observing that $\bold{a}$ is standard establishes Theorem \ref{rigidthm}.

This sequence is picked especially to take advantage of the rigidity of $T$. Section \ref{full sec} shows that for many natural sequences $\bold{b}$ there exists one and the same full measure set such that $\Leb (\LS B(T^ix,b_i))=1$ for every $x$.
\begin{rem} Almost every IET has the property that the orbit of every point is dense. It follows that for almost every IET $\LS B(T^ix,a_i)$ is a dense $G_{\delta}$ set for any $\bold{a}$ with $a_i>0$ for all $i$.\end{rem}



\section{Rauzy-Veech induction}\label{rv ind}
Our treatment of Rauzy-Veech induction will be the same as in \cite[Section 7]{gauss}.  We recall it here. 
 Let $T$ be a $d$-IET with permutation $\pi$. Let $\delta_+(T)$ be the rightmost discontinuity of $T$ and $\delta_-(T)$ be the rightmost discontinuity of $T^{-1}$.
  Let $\delta_{max}(T)=\max\{\delta_+(T),\delta_-(T)\}$. Let $I^{(1)}(T)= [0,\delta_{max}(T))$.
 Consider the induced map of $T$ on $[0,\delta_{\max})$ denoted $T|_{[0,\delta_{\max})}$. If $\delta_+ \neq \delta_-$ this is a $d$-IET on a smaller interval, perhaps with a different permutation.
 We will often write $\delta_-,\delta_+, \delta_{max}, I^{(1)}$ for $\delta_-(T),\delta_+(T), \delta_{max}(T), I^{(1)}(T)$ when there is no confusion.

 If $\delta_{max}= \delta_+$ we say the first step in Rauzy-Veech induction is $a$. In this case the permutation of $R(T)$ is given by
\begin{equation} \pi'(j)= \begin{cases}
 \pi (j) & \quad j \leq \pi^{-1}(d)\\ \pi(d) & \quad j=\pi^{-1}(d)+1 \\ \pi(j-1) & \quad \text{otherwise}

\end{cases}.
\end{equation}
We keep track of what has happened under Rauzy-Veech induction by a matrix $M(T,1)$ where
\begin{equation} M(T,1)[ij]= \begin{cases} \delta_{i,j} & \quad j \leq \pi^{-1}(d)\\
 \delta_{i, j-1} & \quad j>\pi^{-1}(d) \text{ and } i \neq d\\
\delta_{\pi^{-1}(d),j} & \quad i=d \end{cases}.
 \end{equation}
 If $\delta_{max}= \delta_-$ we say the first step in Rauzy-Veech induction is $b$.
 In this case the permutation of $R(T)$ is given by
\begin{equation} \pi'(j)= \begin{cases}
 \pi (j) & \quad \pi(j) \leq \pi(d)\\ \pi(j)+1 & \quad \pi(d) < \pi(j) < d \\ \pi(d)+1 & \quad \pi (j)=d

\end{cases}.
\end{equation}
 We keep track of what has happened under Rauzy-Veech induction by a matrix \begin{equation}M(T,1)[ij]= \begin{cases} 1 & \quad i=d \text{ and }j= \pi^{-1}(d) \\ \delta_{i,j} & \quad \text{ otherwise} \end{cases}.
\end{equation}
The matrices described above depend on whether the step is $a$ or $b$ and the permutation $T$ has. The following well known lemmas which are immediate calculations help motivate the definition of $M(T,1)$.
\begin{lem} \label{one step} If $R(T)=S_{L, \pi}$  then the length vector of $T$ is a real number multiple of $M(T,1)L$. \end{lem}
 Let $M_{\Delta}=M\mathbb{R}_d^+ \cap \mathring{\Delta}_d$. Recall $\mathring{\Delta}_d$ is the interior of the simplex in $\mathbb{R}^d$.
\begin{lem}\label{region} An IET with lengths contained in $M(T,1)_{\Delta}$
 and permutation $\pi$ has the same first step of Rauzy-Veech induction as $T$.
 \end{lem}

We define the $n^{\text{th}}$ matrix of Rauzy-Veech induction by $$M(T,n)=M(T,n-1)M(R^{n-1}(T),1).$$
Likewise, we define $I^{(n)}(T):=I^{(1)}(R^{n-1}(T))$. We will often denote this by $I^{(n)}$.
 It follows from Lemma \ref{region} that for an IET with length vector in $M(T,n)_{\Delta}$ and permutation $\pi$ the first $n$ steps of Rauzy-Veech induction agree with $T$.
 If $M$ is any matrix, $C_i(M)$ denotes the $i^{th}$ column and $C_{max}(M)$ denotes the column with the largest sum of entries.
 Let $|C_i(M)|$ denote the sum of the entries in the $i^{th}$ column. Versions of the following lemma are well known and we provide a proof for completeness.
\begin{lem} \label{proscribed rv} If $M(R^n(T),k)$ is a positive matrix and $L= \frac {C_i(M(T,n+k))}{|C_i(M(T,n+k))|}$ then $S_{L, \pi(T)}$ 
 agrees with $T$ through the first $n$ steps of Rauzy-Veech induction.
\end{lem}
\begin{proof} By Lemma \ref{one step} the length vector for $R^m(S_{L, \pi})$ is $ \frac{C_i(M(R^m(T),n+k-m))}{|C_i(M(R^m(T),n+k-m))|}$ for any $m$ where $R^m(S_{L,\pi})$ is defined.
  By our assumption on the positivity of $M(R^n(T),k)$ the vector $\frac{C_i(M(R^n(T),k))}{|C_i(M(R^n(T),k))|}$ is contained in $\mathring{\Delta}_d$. The lemma follows by  Lemma \ref{region} and induction.
\end{proof}
For the proofs of Lemmas  \ref{sep by tower} and \ref{good is sep} in the next section some knowledge of Rokhlin towers is needed. We recall what is used here. Given $I_j^{(n)}(T)=I_j^{(n)}$ its Rokhlin tower is $\underset{i=0}{\overset{|C_j(M(T,n))|-1}{\cup}} T^i(I_j^{(n)})$. By the construction of Rauzy-Veech induction $T$ acts continuously on this set. Also $T^{|C_j(M(T,n))|}(I_j^{(n)})\subset I^{(n)}$. Rokhlin towers are used to recapture information about $T$ from $R^n(T)$.

Let $\mathfrak{R}(\pi)$ denote the set of permutations one can reach by some string of steps $a$ and $b$ from $\pi$.
Let $\Delta_{\mathfrak{R}(\pi)} $ denote the set of IETs with one of these permutations. The dependence on $\pi$ will be suppressed when there is no confusion.
Let $\measure$ denote Lebesgue measure on $\Delta_{\mathfrak{R}}$.

The next definition does not appear in \cite{gauss} (see \cite{ker}) but is important for the next section.
\begin{Bob}
 A matrix $M$ is called $\nu$-\emph{balanced} if $\frac 1 {\nu} <\frac {|C_i(M)|}{|C_j(M)|}<\nu$ for all $i$ and $j$.
\end{Bob}
 Notice that if $M$ is $\nu $-balanced then $|C_i(M)|>\frac {|C_{max}(M)|}{\nu}$.

\section{Proof of Theorem \ref{full measure thm}}\label{full sec}
The main goal of this section is establishing Theorem \ref{full measure thm} to complement Section \ref{rigid}.  Theorem \ref{full measure thm} is proved by Proposition \ref{full measure prop}, which requires a definition.
\begin{Bob} A standard sequence $\{a_i\}_{i=1}^{\infty}$ is called \emph{2-standard} if there exists $r \geq 2$ such that $$\{r^{i-1}a_{r^i}\}_{i=1}^{\infty}=a_r,ra_{r^2},r^2a_{r^3},...$$ is eventually decreasing.
\end{Bob}
\begin{rem} If $\{ia_i\}_{i=1}^{\infty}$ is eventually decreasing then $\bold{a}$ is 2-standard.
\end{rem}

\begin{prop} \label{full measure prop} There exists a full measure set of IETs $\mathcal{V}$ such that for any 2-standard sequence $\bold{a}$ and any $x \in [0,1)$, $\Leb\left( \LS B(T^i(x),a_i)\right)=1$.
\end{prop}
Proposition \ref{full measure prop} implies Theorem \ref{full measure thm}. This is because if $\{ia_i\}_{i=1}^{\infty}$ is eventually decreasing then $a_i$ is 2-standard. If $\{ia_i\}_{i=1}^{\infty}$ is eventually increasing then some 2-standard sequence is term by term less than it.



Next is a criterion for an IET $T$ to have $\Leb \left(\LS B(T^ix,a_i)\right)=1$ for every 2-standard sequence $\bold{a}$. Lemma \ref{good is sep} and Proposition \ref{step 2} prove that $\measure$-almost every IET satisfies the criterion.

\begin{prop} \label{step 1} If $\{a_i\}_{i=1}^{\infty}$ is 2-standard and $T$ is a $\Leb$-ergodic IET, such that there exists $r>1,c>0,e>0$ and a positive lower density set of natural numbers $k$ such that at least $cr^k$ elements of $\{T^{r^k}x,T^{r^k+1}x,...,T^{r^{k+1}}x\}$ are $\frac e {r^k}$ separated then $\Leb\left(\LS B(T^i(x),a_i)\right)=1$.
\end{prop}

\begin{proof} It suffices by the ergodicity of $T$ to show $\Leb\left(\LS B(T^i(x),a_i)\right)>0$
(Proposition \ref{LS reduction}).
 Assume $\{k_i\}_{i=1}^{\infty}$ is a sequence of positive density satisfying the condition of the proposition.
 As before we want to consider
$$\Leb\left(\underset{i=r^k}{\overset{r^{k+1}}{\cup}} B(T^ix,a_i)\backslash \underset{i=N}{\overset{r^{k}}{\cup}}B(T^ix,a_i)\right)$$
  when $\Leb\left(\underset{i=N}{\overset{r^{k}}{\cup}}B(T^ix,a_i)\right)$  is small.
   However, this approach does not work if $c< \frac 1 r$.
   To work around this we will only pay attention to some of the $k_i$.
   Let $l_1=k_1$ and inductively let
 $$l_{n+1}= \min \{k_i: r^{k_i}>3c^{-1}r^{l_n+1}\}.$$
    Notice that $\{l_i\}_{i=1}^{\infty}$ is a set of positive density.
 Choose $\epsilon<\frac 1 4 c e$.
If $\Leb\left(\underset{i=N}{\overset{r^{l_j}}{\cup}}B(T^i(x),a_i)\right)<\epsilon$ then Lemma \ref{separated} implies
\begin{multline}\lefteqn{\Leb\left(\underset{i=r^{l_{j}}}{\overset{r^{l_{j}+1}}{\cup}} B(T^ix,a_i)\backslash   \underset{i=N}{\overset{r^{l_{j-1}+1}}{\cup}}B(T^ix,a_i)\right)}\\>(cr^{l_{j}}-2r^{l_{j-1}+1}-\frac {\epsilon}{e}r^{l_{j}})a_{r^{l_{j}+1}}>\frac 1 4 c r^{l_{j}+1}a_{r^{l_{j}+1}}.
\end{multline}
Observe that $a_r,ra_{r^2},...$ is a decreasing sequence with divergent sum and thus $\underset{k \in S } {\sum}r^ka_{r^{k+1}}=\infty$ for any set $S$ of positive lower density. This implies that $\Leb \left(\LS B(T^ix,a_i)\right)> \epsilon$ and the proposition follows.
\end{proof}
\begin{rem} This proposition is false if one only assumes that the set of $k$ has positive upper density.
\end{rem}
Next we will show that almost every IET satisfies the hypothesis of the Proposition \ref{step 1}.
\begin{Bob} Given constants $\nu$ and  $e>0$ we say an IET $T$ is $i$-good if:
\begin{enumerate}
\item There exists $n_0$ such that $M(T,n_0)$ is $\nu$-balanced.
\item $|C_{max}(M(T,n_0))| \in [2^i,2^{i+1}]$.
\item Let $T_{n_0}$ denote the IET $R^{n_0}(T)$ (which is defined on $I^{(n_0)}$). For each $x$ the points of
$\{T_{n_0}(x),T_{n_0}^2(x),...,T_{n_0}^{\lceil 20\nu^2d\rceil}(x)\}$
are $\frac {e\Leb(I^{(n_0)})} {20\nu^2d}$ separated.
\end{enumerate}
\end{Bob}
\begin{rem} The definition of $i$-good depends on $\nu$ and $e$ but for readability this is suppressed. One should think that in each Rauzy class we have an appropriate $\nu$ and $e$ but these may change in different Rauzy classes.
\end{rem}
We now proceed with two lemmas which show that the $i$-good condition implies a separation condition of the type in Proposition \ref{step 1}.


\begin{lem} \label{sep by tower} Let $T$ be $i$-good and $n_0,e,\nu$ be as in the definition and $r \neq r'<2^{i}\nu d$ be natural numbers. If $d(T^rx,T^{r'}x)<\frac{e}{20\nu^2d2^{i+1}}$ then
$T^r(x) \in \underset{k=0}{\overset{2^{i+1}}{\cup}}B(T^{k}u,\frac{e}{20\nu^2d2^{i+1}}) $
 and
$T^{r'}x\in \underset{k=0}{\overset{2^{i+1}}{\cup}}B(T^{k}v,\frac{e}{20\nu^2d2^{i+1}})$
where $u,v$ are either discontinuities of $R^{n_0}(T)$ or 0 or $\lambda(I^{(n_0)})$.
\end{lem}
\begin{proof} Let $\epsilon>0$, $\tilde{B}_x(n,\epsilon)=\underset{k=0}{\overset{n}{\cup}}B(T^kx,\epsilon)$ and $$\hat {B}_{m,r}(\epsilon)=\underset{v \text{ a discont of } T_{m}}{\cup} \tilde{B}_{v}(r,\epsilon)$$  and assume that
$$y\notin \hat {B}_{n_0,2^{i+1}}(\epsilon)
\cup \underset{k=0}{\overset{2^{i+1}}{\cup}}B(T^{k}0,\epsilon) \cup
\underset{k=0}{\overset{2^{i+1}}{\cup}}B\left(T^{k}\left(\Leb(I^{(n_0)})\right),\epsilon\right).$$
 (Note for $\epsilon$ not small enough there are no such $y$.) Because $|C_i(M(T,n_0))|<2^{i+1}$ if $k_0=\min\{k:T^{-k}(y)\in I^{(n_0)}\}$ we have that $T^{-k_0}B(y,\epsilon)\subset I_j^{(n_0)}$ for some $j$ and that $T^{-i}$ acts continuously (and thus isometrically) on $B(y,\epsilon)$ for all $0 \leq i<k_0$. Let $\epsilon_0=\frac{e}{20\nu^2d2^{i+1}}$ and if $d(T^rx,T^{r'}x)<\epsilon_0$,
$$
T^rx \notin \hat {B}_{n_0,2^{i+1}}(\epsilon_0) \cup
\underset{k=0}{\overset{2^{i+1}}{\cup}}B(T^{k}0,\epsilon_0) \cup \underset{k=0}{\overset{2^{i+1}}{\cup}}B\left(T^{k}\left(\Leb(I^{(n_0)})\right),\epsilon_0\right)
$$
${k=\min\{0 \leq l:T^{-l}(T^rx)\in I^{(n)}\}}$ and ${k'=\min\{0 \leq
l:T^{-l}(T^{r'}x)\in I^{(n)}\}}$ then
$d(T^{-k}T^rx,T^{-k'}T^{r'}x)=d(T^rx,T^{r'}x)$. This contradicts the
$i$-good assumption because if for instance $r<r'$ we have
$$T^{-k'}T^{r'}x \in
\{T^{-k}T^rx,T_{n_0}(T^{r-k}x),...,T_{n_0}^{\lceil
20\nu^2d\rceil}(T^{r-k}x)\}.$$
\end{proof}

\begin{lem} \label{good is sep} If $T$ is $i$-good then at least $2^i 20\nu d -2(d+1)2^{i+1}$
of the elements of $\{x, Tx,...,T^{\lceil 2^{i+1}20 \nu d\rceil
}x\}$ are at least  $\frac e {20\nu^2d} (2^{i+1})^{-1}$ separated.
\end{lem}
The number of points is positive because $\nu>1$.
\begin{rem} The idea of Lemmas \ref{sep by tower} and \ref{good is sep}  is to make precise a straightforward observation.
 Because $T^i$ is a (continuous) isometry on $I_j^{(n)}$ for $0\leq i<|C_i(M(T,n))|$
 we have that if $y$ is not too close to the early part of the orbit of a singularity and if $x$ is close to $y$
  then $x$ and $y$ pull back under $T$ to close points in $I_j^{(n)}$.
   Therefore the $i$-good condition gives separation for the orbit of points under $T$ away from discontinuities of $T^{2^{i+1}}$ (Lemma \ref{sep by tower}). Therefore if we examine enough points a positive proportion must be separated (Lemma \ref{good is sep}).
\end{rem}
\begin{proof} There are $d-1$ discontinuities of $T_{n_0}$ and
by the $i$-good condition the Rokhlin towers over any sub-interval of
$I^{(n_0)}$ have at most $2^{i+1}$ levels. So there are at most $2^{i+1}(d+1)$
 points that the previous lemma does not rule out being $\frac e {20\nu^2d} (2^{i+1})^{-1}$ separated.
  Also by the $\nu$-balanced condition the Rokhlin towers over any sub-interval of $I^{(n_0)}$ have at least $2^{i}\nu^{-1}$
   levels and therefore the images of $\{T_{n_0}(x),T_{n_0}^2(x),...,T_{n_0}^{\lceil 20\nu^2 d\rceil }(x)\}$ under $T$ before first return to $I^{(n_0)}$ have at least $20\nu^2d(2^i\nu^{-1})=20\nu d2^i$ elements.
\end{proof}


The proof of Proposition \ref{full measure prop} is completed by the following proposition which shows the almost every IET is $i$-good for a positive density set of $i$. By Lemma \ref{good is sep} these IETs satisfy the hypothesis of Proposition \ref{step 1}.
\begin{prop} \label{step 2} There exists a constant $c_{\mathfrak{R}}'>0$ such that for $\measure$-almost every IET $T$, $$\{ i: T \text{ is i-good}\}\subset\mathbb{N}$$ has lower density at least $c_{\mathfrak{R}}'>0$ where $c_{\mathfrak{R}}'$ is a constant depending only on $\mathfrak{R}$ (and $e, \nu$ in the definition of $i$-good).
\end{prop}
To prove this proposition we first establish other results. Kerckhoff proved independence type results for Rauzy-Veech induction that we provide a slight reformulation of \cite[Corollary 1.7]{ker}.
\begin{prop} (Kerckhoff) Let $\mathfrak{R}$ be one of the Rauzy classes of permutations of $d$-IETs. There exist $p>0,K>1$ and $\nu_0>1$ depending only on $\mathfrak{R}$ such that for any matrix of Rauzy-Veech induction $M'=M(S,n)$ we have
\begin{multline}
\lefteqn{\measure(\{T: \pi(T)=\pi(S), T \in M'_{\Delta} \exists m>n \text{ such that } M(T,m) \text{ is }}\\ \nu_0 \text{-balanced and }  |C_{max} (M(T,m))| <K^d|C_{max}(M')|\})> p\measure ( M'_{\Delta})\end{multline}
\end{prop}
This proposition is useful because the constants are independent of $M'$.
\begin{cor} \label{bal den} For $\measure$-almost every IET $T$ the set
 $$\{i: \exists n \text{ such that }M(T,n) \text{ is } \nu_0 \text{-balanced and } |C_{max}(M(T,n))| \in [2^i,2^{i+1}]\}\subset\mathbb{N}$$
  has lower density at least $c_{\mathfrak{R}}>0$ where $c_{\mathfrak{R}}$ is a constant depending only on $\mathfrak{R}$ and $\nu_0$.
\end{cor}
\begin{proof} Consider the independent $\mu$ distributed random variables $F_1,F_2,...$ where $\mu$ takes value 1 with probability $p$ and 0 with probability $1-p$ and $F_i:\Omega \to \{0,1\}$. Recall that one puts a probability measure $\mu^{\mathbb{N}}$ on $\Omega$ such that for any $k\leq n$ and $a_1,...,a_n \in \{0,1\}$ where $k$ of the $a_i$ are 1 we have $$\mu^{\mathbb{N}}(\{t \in \Omega: F_i=a_i \text{ for all } i \leq n\})=p^k(1-p)^{n-k}.$$
By the strong law of large numbers, for $\mu^{\mathbb{N}}$-almost every $t \in \Omega$ we have $\underset{n \to \infty} \lim \frac{\underset{i=1}{\overset{n}{\sum}}F_i(t)}{n}=p$. Let $$G(T) =\{i: \exists n \text{ such that }M(T,n) \text{ is } \nu_0 \text{-balanced and } |C_{max}(M(T,n))| \in [2^i,2^{i+1}] \}.$$ By the previous proposition, given $G(S) \cap [0,N]$ the conditional probability that $N + i  \in G(S)$ for some $0<i \leq \lceil d \log_2(K) \rceil $ is at least $p$. Thus for any natural numbers $n_1,n_2,...,n_k$
\begin{multline}\lefteqn{\measure(\{S: [n_i\lceil d \log_2(K) \rceil, (n_i+1)\lceil d \log_2(K) \rceil] \cap G(S) \neq \emptyset \, \forall i\leq k \})} \\ \geq \mu^{\mathbb{N}}(\{t: F_{n_i}(t)=1 \, \forall i \leq k\}).
\end{multline}
This implies that from \measure  -almost every $T$, $G(T)$ has lower density at least $\frac{p}{ \lceil d \log_2(K) \rceil}$.
\end{proof}
This proposition is useful because of the next result.
\begin{prop} \label{bal} (Kerckhoff \cite[Corollary 1.2]{ker}) If $M$ is $\nu_0$-balanced and $W \subset \Delta_d$ is a measurable set, then $$\frac{\measure(W)}{\measure(\Delta_d)}< \frac{\measure(MW)}{\measure(M\Delta_{d})}(\nu_0)^{-d}.$$
\end{prop}
This proposition is useful because informally what it says is that when $M(T,n)$ is balanced then the conditional probability of the next sequence of steps of Rauzy-Veech induction is proportional to the measure of IETs with that initial sequence of steps of Rauzy-Veech induction.
\begin{cor}\label{often hit} For any measurable $U \subset \Delta_{d-1}$ we have
$$\{i<N:\exists n \text{ with } |C_{\max}(M(T,n))|\in[2^i,2^{i+1}] \text{ and }\frac{L(R^nT)}{|L(R^n(T))|} \in U\}$$
 has density at least $c_{\mathfrak{R}}\nu_0^{-d}\measure(U)$
where $c_{\mathfrak{R}}$ and $\nu_0$ are as in Corollary \ref{bal den}.
\end{cor}
\begin{proof} It follows from Corollary \ref{bal den} that for almost every IET $T$ there exists $\{i_r\}_{r=1}^{\infty} \subset \mathbb{N}$ of lower density at least $c_{\mathfrak{R}}$ and $\{n_r\}_{r=1}^{\infty} \subset \mathbb{N}$ such that
$$M(T,n_r) \text{ is } \nu_0 \text{-balanced and } |C_{max}(M(T,n_r))| \in [2^{i_r},2^{i_r+1}].$$
It follows from the proportional independence provided by Proposition \ref{bal}
(analogously to the proof of  that Corollary \ref{bal den})
 that $\measure$-almost surely a set of $r$ of lower density at least $\nu_0^{-d}\measure(U)$ have $\frac{L(R^{n_r}T)}{|L(R^{n_r}(T))|} \in U$.
\end{proof}
\begin{proof}[Proof of Proposition \ref{step 2}] Let $U_{e}\subset \Delta_{\mathfrak{R}}$ be the set of IETs $S$
where the elements of $\{Sx,S^2x,...,S^{\lceil 20\nu^2d\rceil }x\}$
are $\frac {e}{20\nu^2d}$ separated for all $x \in [0,1)$.
 It is easy to see that for small enough $e$ we have $\measure(U_e)>0$. Notice that if
  $|C_{max}(M(T,n))|\in [2^i,2^{i+1}]$
 then $I^{(n)}(T)>\frac{1}{2^{i+1}}$. Therefore if $R^n(T)$ is $\nu$-balanced and $|C_{max}(M(T,n))|\in [2^i,2^{i+1}]$ and $\frac{L(R^n(T))}{|L(R^n(T))|} \in U_{e}$ then $T$ is $i$-good. The proposition now follows from Corollary \ref{often hit}.
\end{proof} 

We have established Theorem \ref{full measure thm}, but one can also establish the dual formulation. By similar arguments and Lemma \ref{separated 2} it follows that there exists a full measure set of IETs $\mathcal{V}$ such that for any $\bold{a}$ standard and $\{ia_i\}_{i=1}^{\infty}$ monotone, $T \in \mathcal{V}$ we have $\Leb \left(\LS T^{-i}(B(y,a_i))\right)=1$ for every $y$.

There are similar versions of Theorem \ref{full measure thm} and the preceding comment for almost every direction of almost every flat surface. This follows by Fubini's theorem and a parallel argument to the proof of Corollary \ref{kurz for flow}.
\section{concluding remarks}

We established that for any $\{a_i\}_{i=1}^{\infty}$ and flat surface $Q$ almost every direction satisfies that $B(y,a_i)$ is Borel-Cantelli for any $y$. Moreover, any $x$ is in $\LS T^{-i}(B(y,a_i))$ for almost every $y$. In \cite{tseng} it was shown that this can not be improved to be a statement about every pair $(x,y)$. In fact, for rotations ($Q$ the torus) the set of $y$ such the $\underset{i \to \infty} {\liminf} \, i|y-R_{\alpha}^i(x)|>0$ is a set of Hausdorff dimension 1 
 for any $x$ and $\alpha$.


Likewise, Theorem \ref{rigidthm} can not be improved to be a statement about every IET. There are many IETs that satisfy MSTP, in particular Pseudo-Anosov IETs. (Recall that an IET is Psuedo-Anosov if it is fixed up to rescaling by a power of Rauzy-Veech induction.) This follows from the fact that they are linearly recurrent and by modifying Kurzweil's proof that rotation by a badly approximable number satisfies MSTP. It also follows from \cite[Theorem 1]{BC}. A particular case of this is given by any minimal IET which has its lengths chosen over the same quadratic number field \cite{bosh-car}. For IETs MSTP also survives inducing on sub-intervals of $[0,1)$. 
This implies that the induced map of a rotation by a badly approximable number gives a 3-IET satisfying MSTP. Therefore, there are IETs that satisfy MSTP and have $\underset{n \to \infty}{\liminf}n \, e(n)=0$ (one can see this by inducing a rotation by a badly approximable number on a generic interval). For rotations this does not happen. 
\begin{ques} Fix $x$ and $T$.  Does the set $\{y: \underset{i \to \infty}{\liminf} \, i|T^ix-y|>0\}$ have Hausdorff dimension 1?
\end{ques}
\begin{ques} Does there exist a (not necessarily decreasing) sequence $\{a_i\}_{i=1}^{\infty}$ with divergent sum and a positive measure set of IETs $M$, such that for all $T \in M$, ${\Leb(\LS B(T^i(x),a_i))=0}$ for almost every $x$?
\end{ques}
Such a sequence does not exist for rotations. This fact follows from Kurzweil's proof of the first part of Theorem \ref{startthm}.
\begin{ques} Let $\{y_i\}_{i=1}^{\infty} $ be a sequence of points is $ [0,1)$ and $\{a_i\}_{i=1}^{\infty}$ be a sequence of positive real numbers with divergent sum. Is it true that for almost every IET $T$, we have $\Leb(\LS T^{-i}(B(y_i,a_i)))=1$?
\end{ques}
This is true for rotations. This fact follows from Kurzweil's proof of the first part of Theorem \ref{startthm}.
\section{Acknowledgments}
I would like to thank my advisor, M. Boshernitzan, for many helpful discussions. Much of this was inspired and enabled by our joint work in \cite{BC} and \cite{BCIET}. I would like to thank J. Athreya, D. Kleinbock, L. Marchese, C. Ulcigrai, and W. Veech for helpful discussions. I would like to thank the organizers of Dynamique dans l'espace de Teichmueller in Roscoff France in June 2008. I would like to thank the referees for many suggestions that greatly improved the paper. I was supported in part by Rice University's Vigre grant DMS-0739420 and a Tracy Thomas award.

\end{document}